\documentclass[12pt, reqno]{amsart}

\usepackage{epsfig}

\newcommand{\kahler}{K\"ahler\ }

\newcommand{\PP}{{\mathbb P}}
\newcommand{\R}{{\mathbb R}}
\newcommand{\C}{{\mathbb C}}

\newcommand{\Z}{{\mathbb Z}}
\newcommand{\N}{{\mathbb N}}

\renewcommand{\d}{\partial}
\newcommand{\dbar}{\bar\partial}
\newcommand{\ddbar}{\partial\dbar}

\newcommand{\go}{\mathfrak}

\newcommand{\al}{\alpha}

\newcommand{\ep}{\varepsilon}

\newtheorem{theo}{{\sc Theorem}}[section]
\newtheorem{cor}[theo]{{\sc Corollary}}
\numberwithin{equation}{section}

\theoremstyle{plain}

\newtheorem{lem}[theo]{Lemma}

\newenvironment{rem}{\medskip\noindent{\it Remark:\/} }{\medskip}
\newenvironment{defin}{\medskip\noindent{\it Definition:\/} }{\medskip}

\allowdisplaybreaks

\title[Asymptotic slopes of the Aubin-Yau functional]
{Asymptotic slopes of the Aubin-Yau functional and calculation of the Donaldson-Futaki invariant}

\author{Daniel Rubin}
\thanks{The author was partially supported by NSF grant DMS-12-66033.}
\address{Department of Mathematics, Cornell University, Ithaca, NY 14850}
\email{drubin@math.cornell.edu}

\begin{document}

\maketitle

\begin{abstract}
We derive an explicit formula for the asymptotic slope of the Aubin-Yau functional along a Bergman geodesic on a surface of complex dimension 2, extending the work of Phong-Sturm \cite{PSMab} on Riemann surfaces. This is equivalent to an explicit calculation of the Donaldson-Futaki invariant of a test configuration. The slope is given as a rational linear combination of period integrals of rational functions that sum to a rational number. The result gives a way to check directly whether a two dimensional projective variety is $K$-stable.
\end{abstract}

\section{Introduction}
		
This paper begins an investigation into the relationship between certain algebraic notions of stability and the (possibly singular) solutions of equations of canonical metrics on \kahler manifolds via asymptotic analysis of certain convex energy functionals.

\smallskip 	
		
The existence of canonical metrics in \kahler geometry is now well understood to be linked to certain notions of stability in the sense of geometric invariant theory. Historically, the first result of this type was the theorem of Donaldson-Uhlenbeck-Yau on the equivalence of the existence of Hermitian-Einstein metrics on holomorphic vector bundles with Mumford-Takemoto stability \cite{D1,UY}. Following Yau's seminal proof of the Calabi conjecture \cite{Y1}, the recent resolution of the Yau-Tian-Donaldson conjecture establishes the equivalence of $K$-polystability and the existence of K\"{a}hler-Einstein (KE) metrics on Fano manifolds \cite{Y,CDS1,CDS2,CDS3,DS,T2,Berm}. Several questions remain unresolved, however, including: 
\begin{itemize}
\item[(Q1.)] Given $(X,L)$ a very ample line bundle over a \kahler manifold, can we determine if $(X,L)$ is $K$-stable?
\item[(Q2.)] What are the possible singularities of generalized plurisubharmonic \kahler potential solutions to the KE equation in the Fano case?
\end{itemize}
We take a variational approach to attempt to address these questions. The existence of KE metrics and the more general  existence of constant scalar curvature \kahler (cscK) metrics are determined by the properness of certain energy functionals (see \cite{D2}, for example). It is therefore desirable to understand the link between algebraic notions of stability and the behavior of these functionals. 

\smallskip

In this article we consider certain special one-parameter degenerations of the \kahler class along which the energy functionals restrict to become convex functions. Along these directions, the relevant energy functionals have asymptotic slopes that are related to algebraic stability invariants, and determine the properness of the energy. For instance, the existence of a degeneration along which the asymptotic slope is negative is an obstruction to the existence of a minimizer for the functional. Asymptotics of energy functionals are also of considerable interest in the study of partition functions over Bergman metrics \cite{KZ}. The aim of this article is to establish the asymptotics of the Aubin-Yau functional, and express its asymptotic slopes as an explicit formula in terms of the data of test configurations by means of analysis of singular integrals. 
As a result, we obtain a way of checking the Chow-Mumford or $K$-stability of a variety.

\smallskip

Let $(X,\omega_0)$ be a \kahler manifold of complex dimension $n$ with reference \kahler metric $\omega_0$, $\omega_\phi = \omega_0 + \frac{\sqrt{-1}}{2\pi}\ddbar \phi$, $Ric(\omega_0)=-\frac{\sqrt{-1}}{2\pi}\log \omega_0^n$ the Ricci form of $\omega_0$, and $V = \int_X \omega_0^n$. The functionals described below are defined on the space of \kahler potentials
\begin{equation}
	\mathcal{K} = \{\phi \in C^\infty(X), \omega_0 + \frac{\sqrt{-1}}{2\pi}\ddbar \phi > 0\}.
\end{equation}



\smallskip

\begin{defin}
 The \textit{Aubin-Yau functional} $F_{\omega_0}^0(\phi)$ is given by
 \begin{equation}
 	F_{\omega_0}^0(\phi) = \frac{1}{n+1}\frac{1}{V}\int_X \phi \sum_{i=0}^n \omega_0^i \wedge \omega_{\phi}^{n-i}.
 \end{equation}
\end{defin}

The significance of the Aubin-Yau functional in \kahler geometry is discussed extensively in \cite{PSLSC}. Minimizers of $F_{\omega_0}^0(\phi)$ in the space of Bergman metrics are called \emph{balanced metrics}. Zhang \cite{Z} proved that the existence of a balanced metric is equivalent to Chow-Mumford stability (see also \cite{PSStab}). Donaldson showed that the existence of a cscK metric implies the existence of a balanced metric, hence the Chow-Mumford stability and the existence of a minimizer for $F_{\omega_0}^0(\phi)$ in the space $\mathcal {K}_k$. In this case, the asymptotic slope of the Aubin-Yau functional is necessarily positive. 

\smallskip

The Aubin-Yau functional also relates to $K$-stability. It is also shown in \cite{PSLSC} that the Donaldson-Futaki invariant of a test configuration is equal to a limit of the asymptotic slopes of Aubin-Yau along Bergman geodesics; see Theorem \ref{DFinvariant} below.

\smallskip

Briefly, let us say how the Aubin-Yau functional is related to other functionals in the literature. It is related to the $J$-functional
\begin{equation}
	J_{\omega_0}(\phi) = \frac{\sqrt{-1}}{2\pi V} \int_X \sum_{i=0}^{n-1} \frac{(i+1)}{(n+1)} \d\phi\wedge \dbar\phi \wedge \omega_\phi^{n-i-1}\wedge\omega_0^i
\end{equation}
by
\begin{equation}
	F_{\omega_0}^0(\phi) = \frac{1}{V}\int_X \phi \omega_0^n - J_{\omega_0}(\phi).
\end{equation}
In the special case $[\omega_0] = [K_X^{-1}]$, $F_{\omega_0}^0$ is related to the functional $F_{\omega_0}(\phi)$ by
\begin{equation}
	F_{\omega_0}(\phi) = -F_{\omega_0}^0(\phi) - \log\left(\frac{1}{V}\int_X e^{h_{\omega_0}-\phi}\omega_0^n\right), \qquad Ric(\omega_0)-\omega_0 = \frac{\sqrt{-1}}{2\pi}\ddbar h_{\omega_0}.
\end{equation}
Minimizers of $F_{\omega_0}(\phi)$ are K\"{a}hler-Einstein metrics. Its asymptotics are discussed to establish the necessity of $K$-stability for existence of a KE metric and the issue of uniqueness of KE metrics in \cite{Bern,Berm}. Note the surprising sign on the $F_{\omega_0}^0$ term, given that both $F_{\omega_0}^0$ and $F_{\omega_0}$ are convex along the Bergman geodesics we will define below.

\medskip

Here is the setup for the degenerations we will consider: Let $L \rightarrow X$ be a very ample line bundle, with $S=\{S_0,...,S_{N}\}$ a basis of sections of $H^0(X,L)$ furnishing a Kodaira embedding
\begin{equation}
	X \ni z \mapsto \iota_S (z) = [S_0(z),...,S_N(z)] \in \PP^N.
\end{equation} 
Then the line bundle $L$ is the pullback of the restriction to $\iota(X)$ of the hyperplane bundle $\mathcal{O}_{\PP^N}(1)$. We consider the action of one-parameter subgroups $\sigma_t \in SL(N+1,\C)$ acting diagonally as
\begin{equation}
	\sigma_t \cdot S = (t^{a_0}S_0,...,t^{a_N}S_N), \qquad a_0+...+a_N=0.
\end{equation}
Under this action, $X$ acquires a corresponding family of \kahler metrics
\begin{equation}
	\omega_t = \frac{\sqrt{-1}}{2\pi} \ddbar \log \|\sigma_t \cdot S \|^2, 
		\qquad \|\sigma_t \cdot S\|^2 = \sum_{j=0}^N |t|^{2a_j}|S_j|^2
\end{equation}
which are the restrictions to $\sigma_t \cdot \iota(X)$ of the Fubini-Study metric on $\PP^N$. Written in terms of potentials, we have $\omega_t = \omega_0 + \frac{\sqrt{-1}}{2\pi}\ddbar \phi$, where our reference metric is $\omega_0 = \frac{\sqrt{-1}}{2\pi}\ddbar \|S\|^2$, and 
\begin{equation}
	\phi = \log \frac{\|\sigma_t \cdot S\|^2}{\|S\|^2} = \log \frac{\sum_{j=0}^N
		|t|^{2a_j}|S_j|^2}{\sum_{j=0}^N |S_j|^2}.
\end{equation}
 The finite dimensional space of such potentials as the basis of sections varies is called the \emph{Bergman space} $\mathcal{K}_1$. We may also consider larger Bergman spaces $\mathcal{K}_k$ as we consider powers of the line bundle $L^k$ with larger bases of sections. Note that if $\phi$ is a potential in $\mathcal{K}_1$, then $k\phi$ is a potential in $\mathcal{K}_k$, and furthermore, 
\begin{equation}\label{scaling}
	F_{k\omega_0}^0(k\phi) = kF_{\omega_0}^0(\phi),
\end{equation}
so for our purposes it suffices to look at a single line bundle $L$.

\smallskip

We may assume that $a_0\geq ...\geq a_N$, and we call the sections with weight equal to $a_N$ \emph{sections of lowest weight}. The path $t\mapsto \phi$ defined above in the space of \kahler potentials is called a \emph{Bergman geodesic}.

\medskip

It is known that along such a one-parameter subgroup, $F_{\omega_0}^0(\phi)$ is convex in $u = \log(1/|t|)$. We aim to describe the asymptotic behavior of $F_{\omega_0}^0$ as $u\rightarrow \infty$, or equivalently, as $|t| \rightarrow 0$, and in particular, to determine its asymptotic slope. The asymptotic slope $\mu$ may thus be defined either as
\begin{equation}
\mu = \lim_{u\rightarrow \infty} \frac{d}{du} F^0_{\omega_0}(\phi),
\end{equation}
or as
\begin{equation}
F^0_{\omega_0}(\phi) = \mu \log(\frac{1}{|t|}) + O(1)
\end{equation}
as $|t| \rightarrow 0$. We employ analysis to establish that the singular behavior of the functional is $O(\log|t|)$, and use some algebra to determine the precise coefficient. 

\smallskip

Let us also recall the relation of $F^0_{\omega_0}$ to $K$-stability. Let $F$ be the Donaldson-Futaki invariant of the test configuration, whose sign determines $K$-stability (see the survey \cite{PSLSC} for the definitions and the equivalence of test configurations and Bergman geodesics/one-parameter subgroups). Let $\mu_k$ be the asymptotic slope of $\frac{1}{k}F_{k\omega_0}^0(k\phi)$ in $\mathcal{K}_k$, and let $F$ be the Donaldson-Futaki invariant of the test configuration corresponding to $\phi$ (see the survey \cite{PSLSC} for the definitions and the equivalence of test configurations and Bergman geodesics/one-parameter subgroups). The sign of $F$ determines the $K$-stability of the test configuration. Letting $V_k$ be the volume of the $k$-th scaling, we have
\begin{theo}[\cite{PSLSC}, Lemma 6]\label{DFinvariant}
\begin{equation}
F = \lim_{k\rightarrow \infty} \frac{\mu_kV_k}{k^n}.
\end{equation}
\end{theo}
From the scaling relation (\ref{scaling}), it is clear that the single asymptotic slope $\mu$ computes the Donaldson-Futaki invariant.

\medskip

The following theorem on the asymptotic slope in complex dimension $n=1$ is in \cite{P}:
\begin{theo}\label{n=1}
Let $X$ be a Riemann surface, with $L$ and $\phi$ as above. Then
\begin{equation}\label{n=1slope}
	F_{\omega_0}^0(\phi) = \{-2a_N - \frac{1}{V}\sum_{\text{zeroes of $S_N$}}\sum_{\al=1}^M p_\al^2(m_\al -m_{\al+1})\}\log\frac{1}{|t|} + O(1)
\end{equation}
as $t\rightarrow 0$, where $p_\al,m_\al$ refer to the data of the Newton polygon.
\end{theo}
In this case, the slopes of the Newton polygon $m_\alpha$ are such that $m_\alpha > m_{\alpha+1}$, so the expression in braces above is positive, and the formula gives another proof of the Chow-Mumford stability of curves. Another proof of Theorem \ref{n=1} is given in \cite{PSMab}, along with the slope of the Mabuchi functional. Our approach is inspired by \cite{PSMab}, as well as earlier works on asymptotics of oscillatory integrals in \cite{PhStein,PSAlg}.  

\smallskip

We derive an analogous formula for the Aubin-Yau functional in complex dimension $n=2$ using a similar asymptotic calculation of singular integrals with certain modifications in order to deal with the new complications in higher dimension. We expect that the approach is valid in all dimensions with analogous formulas for the slope, but in this paper we stick to dimension 2 for concreteness and ease of notation. 

\medskip

Here are the results of the paper: First, we observe that for all the integrals that do not involve the highest power of $\omega_\phi$, the entire contribution to the slope is from the lowest weight:
\begin{theo}\label{lowestWeight}
Assume $X$ has dimension $n=2$, with $L$ and $\phi$ as above. For $1\leq i \leq n$,	
	\begin{equation}	
	\frac{1}{V}\int_X \phi \omega_0^i \wedge\omega_\phi^{n-i} = -2a_N\log\frac{1}{|t|} +O(1)
	\end{equation}
as $t\rightarrow 0$.
\end{theo}

\medskip

The main result of this paper is the following formula for the slope:
\begin{theo}\label{slope}
Assume $X$ has dimension $n=2$, with $L$ and $\phi$ as above. Then we have
\begin{equation}
	F_{\omega_0}^0(\phi_t) = \mu \log\frac{1}{|t|} + O(1)
\end{equation}
as $|t|\rightarrow 0$ where the asymptotic slope $\mu$ is given by
\begin{align}\label{formula}
	\mu = (-2a_N)-\frac{1}{3V}&\sum_{Sing(\tilde{D})\cap Z(\tilde{S_N})}
		\sum_{\substack{\text{faces $F_c$}\\ \text{of $\go{N}$}}}16d_c \nonumber\\
		&\sum_{\{i,j,k,l\}^*}D_4(i,j,k,l)\int_0^\infty \int_0^\infty \frac{ x^{2(p_i+p_j+p_k+p_l)-1}y^{2(r_i+r_j+r_k+r_l)-1}}{(\sum_\al x^{2p_\al}y^{2r_\al})^4}dxdy
\end{align}
where $\go{N}$ is the Newton polytope of the data at a singular point with normal crossings, the exponents $p_i,r_i$, etc. refer to the data of the Newton diagram, and the sum indicated by $\{i,j,k,l\}^*$ is over sets of four indices $\{i,j,k,l\}$ corresponding to an unordered selection of four points of the Newton diagram, all lying on the face $F_c$ (not necessarily vertices of $F_c$), at least three of which are distinct, and not all collinear. The sum over $\al$ is over indices corresponding to all points of the Newton diagram lying on the face $F_c$. The term $d_c$ is defined by describing the equation of the face $F_c$ in $(p,r,q)$-space as
\begin{equation}
	F_c = \{a_c p + b_c r + q = d_c\}\cap \go{N},
\end{equation}
and is positive. The positive, symmetric, integer-valued function denoted $D_4(i,j,k,l)$ is a sum of Gram determinants depending on the vectors $(p_i,r_i),(p_j,r_j),(p_k,r_k),(p_l,r_l)$, and is defined in (\ref{D4def}) below. The integrals in the formula are all convergent.
\end{theo}

\begin{rem}
We note that as written, the asymptotic slope $\mu$ is a difference of positive terms: a positive trivial contribution $-2a_N$ from the lowest weight, minus the positive nontrivial contribution. The positive nontrivial contribution has a complicated dependence on the weights since it jumps as the shape of the Newton polytope changes. Note however that the slope is linear and homogeneous in the weights $a_i$ or $q_i$, at least for fixed geometries of the Newton polytope.
\end{rem}

\medskip

The outline of the paper is as follows: in section two, we describe the proof of Theorem \ref{slope}. We begin by isolating the contribution from the lowest weight. We then calculate the lowest order terms that appear in the volume forms, making use of some algebraic identities that give us the Gram determinant quantities $D_4(i,j,k,l)$. At this point, we introduce and describe the important features of the Newton diagram associated to a one-parameter subgroup, and carry out the computation of the singular part of the integrals. In the following section, we carry through the slope calculation for some simple examples. To conclude we outline some directions for further work.

\medskip
\textbf{Acknowledgements:} I would like to thank my Ph.D. advisor D.H. Phong for his guidance and helpful insight into the problem, and also Steve Zelditch for his encouragement. Also thanks to Karsten Gimre for help in checking calculations during the preparation of this work and for his extensive assistance with the use of Mathematica.

\section{Details of the Slope Calculation}

It is convenient to utilize the notation from \cite{PSMab} and isolate the lowest power of $|t|$ as follows:
\begin{equation}
\phi = \log \frac{|\sigma S|^2}{|S|^2} - 2a_N\log\frac{1}{|t|}, \quad 
	|\sigma S|^2 = \sum_{j=0}^N |t|^{2q_j}|S_j|^2, \quad |S|^2 = \sum_{j=0}^N |S_j|^2, 
\end{equation}
where the exponents
\begin{equation}
	q_j = a_j - a_N \geq 0
\end{equation}
are the \textit{non-negative weights}. Note that at least one of the non-negative weights is equal to 0. By the assumption that the basis of sections furnishes a smooth Kodaira embedding, there is no point on $X$ where all of the sections vanish. This implies that $\log |S|^2$ is bounded on $X$, and therefore
\begin{equation}
	\int_X \log |S|^2 \omega_0^i \omega_\phi^{n-i} \leq CV = O(1),
\end{equation}
so we may drop this term from $\phi$ for the calculation of the asymptotic slope. It is then trivial to compute the contribution to the slope from the section of lowest weight, since
\begin{align}
	\frac{1}{V}\int_X -2a_N\log\frac{1}{|t|}\omega_0^i \omega_\phi^{n-i} &= -2a_N\log\frac{1}{|t|}\frac{1}{V}\int_X\omega_0^i \omega_\phi^{n-i}\nonumber \\
	&= -2a_N\log\frac{1}{|t|}
\end{align}
for each $0\leq i \leq n$. These $n+1$ terms account for the overall contribution of $-2a_N$ (which is non-negative since $a_N\leq 0$) to the asymptotic slope. 

We set out to determine the nontrivial contribution to the slope, that is, to compute the singular part of 
\begin{equation}
	A_i(t) = \int_X \log|\sigma S|^2 \omega_0^i \omega_\phi^{n-i}.
\end{equation}
Here is the basic idea: The singular part of the global integral $A_i(t)$ may be calculated by integrating only over neighborhoods of isolated points, namely the transverse intersection points of the zero divisor of the section(s) of lowest weight with itself and the zero divisors of the other sections.

\begin{proof}
Observe first that the integrand of $A_i(t)$ is bounded away from the union of the zero sets of the sections of lowest weight. Let $s_N$ be a section of lowest weight. Suppose that in a neighborhood of a smooth point $p = 0$ of $\{S_N=0\}$, we may take complex coordinates $z_1,z_2$ (possibly after a resolution) in which $\{S_N=0\}=\{z_1=0\}$, and each of the other sections in this trivialization may be written in the form $s_i = z_1^{p_i}u_i(z_1,z_2)$, where $u_i$ is a unit. Then it will follow from the calculations below that the volume forms $\omega_0^2$, $\omega_0\wedge\omega_\phi$, and $\omega_\phi^2$ only contain terms of strictly higher order in $|z_1|$ and $|z_2|$, and thus there is no contribution to the $\log|t|$ term in $F_{\omega_0}^0$ by Lemma \ref{slopeCalc}. 
\end{proof}

In general, we recall Hironaka's result on resolution of singularities: There exists a resolution $\mu:\tilde{X} \rightarrow X$ such that $\mu^*D+Exc(\mu)=\tilde{D}$ has simple normal crossing support. On $\tilde{X}$, the nontrivial contributions to the slope come from a finite set of points of intersections with the other divisors with $\mu^*D_N$.

Assume that we have a set of coordinates in a neighborhood of a point, taken to be the origin, at which the sections vanish with normal crossings. This means that we our sections $S_j$ are written in these coordinates as
\begin{equation}
S_j = x^{p_j}y^{r_j}u_j(x,y),
\end{equation}
where the $u_j(x,y)$ are holomorphic functions that do not vanish at the origin. 
For a more detailed account of Hironaka's theorem and its use in the analysis of integrals see \cite{PSAlg}.

\subsection{Algebraic computation of the volume form to lowest order}

We must first compute $\omega_\phi$, $\omega_0\wedge\omega_\phi$, and $\omega_\phi^2$. We find
\begin{align}
\omega_\phi &= \frac{\sqrt{-1}}{2\pi}\d \dbar \log |\sigma S|^2 \\
	&= \frac{\sqrt{-1}}{2\pi|\sigma S|^4} \sum_{i,j} |t|^{2q_i+2q_j}\left( |S_j|^2\d S_i\wedge \dbar \bar{S_i} - S_i \bar{S}_j \d S_j\wedge \dbar \bar{S}_i\right)\\
	&= \frac{\sqrt{-1}}{2\pi|\sigma S|^4}\sum_{i,j} |t|^{2(q_i+q_j)}|u_i(0)|^2|u_j(0)|^2|x|^{2(p_i+p_j-1)}|y|^{2(r_i+r_j-1)}\nonumber\\
	& \left( (p_i^2-p_i p_j)|y|^2dx\wedge d\bar{x} + (p_ir_i-p_jr_i)\bar{x}ydx\wedge d\bar{y} + (p_ir_i - p_ir_j)x\bar{y}dy\wedge d\bar{x} + (r_i^2-r_jr_i)|x|^2dy\wedge d\bar{y}\right)\nonumber\\
	& + O(...).
\end{align}
Here by $O(...)$ we mean all higher order terms in $|x|$ and $|y|$. Taking the wedge product yields
\begin{align}
\omega_{\phi}^2 &= \frac{1}{4\pi^2|\sigma S|^8}\sum_{i,j,k,l} |t|^{2(q_i+q_j+q_k+q_l)}|x|^{2(p_i+p_j+p_k+p_l-1)}|y|^{2(r_i+r_j+r_k+r_l-1)}|u_i|^2|u_j|^2|u_k|^2|u_l|^2 \nonumber\\
	& 2\left[ (p_i^2 - p_ip_j)(r_k^2 - r_kr_l) - (p_ir_i-p_jr_i)(p_kr_k-p_kr_l)\right]\sqrt{-1}dx\wedge d\bar{x}\wedge \sqrt{-1}dy\wedge d\bar{y}\nonumber\\
	& + O(...),
\end{align}
where we have used the symmetry in the indices $(i,j)\leftrightarrow(k,l)$ to obtain twice the quantity in brackets. The quantity in the brackets we will denote by the symbol $(ijkl)$, and it may be simplified as

\begin{align}
(ijkl) &= (p_i^2 - p_ip_j)(r_k^2 - r_kr_l) - (p_ir_i-p_jr_i)(p_kr_k-p_kr_l)\nonumber\\
	&= (p_i r_k - p_k r_i)(p_i-p_j)(r_k-r_l).
\end{align}
The other volume forms are the same except for the factors of $|t|$ that only appear in $\omega_\phi$:
\begin{align}
\omega_0\wedge\omega_{\phi} =& \frac{1}{4\pi^2|S|^4|\sigma S|^4}\sum_{i,j,k,l} |t|^{2(q_i+q_j)}|x|^{2(p_i+p_j+p_k+p_l-1)}|y|^{2(r_i+r_j+r_k+r_l-1)}|u_i|^2|u_j|^2|u_k|^2|u_l|^2 \nonumber\\
	& \left[(ijkl)+(klij)\right]\sqrt{-1}dx\wedge d\bar{x}\wedge \sqrt{-1}dy\wedge d\bar{y}\nonumber\\
	& + O(...).
\end{align}

We will see later that it is sufficient to consider only these lowest order terms. From now on, we will also assume that $u_j(0)=1$ for $j=0,...,N$. 

\smallskip
 
The appearance of the determinant-like quantity $(ijkl)$ is a novel feature in dimension $n>1$. Since we get the same monomial in the numerator when a single set of indices $\{i,j,k,l\}$ are picked from among the four summations, the overall coefficient on each term is a sum of the quantities $(ijkl)$ over all permutations of their order. This creates a certain amount of cancellation. In particular, it rules out terms where the same index is taken in each of the four sums. Let us make some simple observations about the symbol $(ijkl)$. First, $(ijkl)=0$ if $i=j$, $k=l$, or $i=k$. If $i$ and $j$ are distinct indices, the only possibly non-zero symbols involving only $i$ and $j$ are $(ijji)$ and $(jiij)$. But
\begin{align}
	(ijji) &= (p_i r_j - p_j r_i)(p_i - p_j)(r_j - r_i)\nonumber\\
		&= -(p_j r_i - p_i r_j)(p_j - p_i)(r_i - r_j)\nonumber\\
		&= -(jiij),
\end{align}
so $(ijji)+(jiij)=0$, and therefore there are no terms in the lowest order part of the volume form with only two distinct indices taken from the sum.  

\smallskip

Now consider the case of three distinct indices:
\begin{lem}\label{3VertexFace}
The nonzero symbols $(ijkl)$ consisting of a set of three indices with one repeated have sum
\begin{align}
	(ijki)+(ikji)+(jiik)+(kiij)+(jiki)+(kiji) &= (p_j r_i - p_k r_i - p_i r_j + p_k r_j + p_i r_k - p_j r_k)^2\\
	&= ([ij]-[ik]+[jk])^2,
\end{align}
where $[ij] = p_i r_j-p_j r_i$.
\end{lem}

\begin{rem}
The quantity $([ij]-[ik]+[jk])^2$ represents the square of the area of any parallelogram with three vertices $\{(p_i,r_i),(p_j,r_j),(p_k,r_k)\}$, and is nonnegative and symmetric in the indices $i,j,k$. This quantity is also known as the \emph{Gram determinant} of the difference vectors between any two of the vectors $\{(p_i,r_i),(p_j,r_j),(p_k,r_k)\}$ and the third one.
\end{rem}

And the remaining case of four distinct indices:
\begin{lem}
The summation of symbols $(ijkl)$ over all permutations of 4 distinct indices yields
\begin{equation}
	\sum_{\sigma \in S_4}(\sigma(i)\sigma(j)\sigma(k)\sigma(l)) = ([ij]-[ik]+[jk])^2+([jk]-[jl]+[kl])^2+([kl]-[ki]+[li])^2+([li]-[lj]+[ij])^2
\end{equation}
\end{lem}
These two algebraic identities, obtained by brute trial and error, may be verified quickly by a computer.

\smallskip

As a corollary, it is clear from these formulas that the lowest order terms in the volume form are non-negative, and are equal to 0 if the four indices correspond to collinear points in the $(p,r)$-plane. 
We set
\begin{equation}
	D_3(i,j,k) = ([ij]-[ik]+[jk])^2
\end{equation}
and
\begin{eqnarray}\label{D4def}
	D_4(i,j,k,l) =
	\begin{cases}D_3(i,j,k)+D_3(j,k,l)+D_3(k,l,i)+D_3(l,i,j) & \text{if all indices distinct}\\
	\frac{1}{2}\left(D_3(i,j,k)+D_3(j,k,l)+D_3(k,l,i)+D_3(l,i,j)\right) & \text{if any two indices are the same}.
	\end{cases}
\end{eqnarray}
The factor of $1/2$ in the case of a repeated index compensates for the overcounting by transposing the slots of the repeated index, and so $D_4(i,j,k,k)=D_3(i,j,k)$. 

\smallskip

We may thus rewrite the lowest-order part of $\omega_\phi^2$ as a sum of positive terms as 
\begin{align}
	\omega_\phi^2 = \frac{1}{|\sigma S|^8}\sum_{\{i,j,k,l\}^*}&\bigg[2D_4(i,j,k,l)|t|^{2(q_i+q_j+q_k+q_l)}|x|^{2(p_i+p_j+p_k+p_l)-2} |y|^{2(r_i+r_j+r_k+r_l)-2}\nonumber\\
	&\frac{\sqrt{-1}}{2\pi}dx\wedge d\bar{x} \frac{\sqrt{-1}}{2\pi}dy\wedge d\bar{y}\bigg] + O(...)
\end{align}
and similarly for $\omega_0^2$ and $\omega_0\wedge \omega_\phi$.

\subsection{Newton diagram}

The analysis of the singular integrals appearing in the Aubin-Yau functional is well facilitated by appealing to the geometry of the Newton polytope.

\begin{defin}
We call the set of points $\{(p_i,r_i,q_i)\in \R^3_+\}_{i=0}^N$ the \emph{Newton diagram} of the data. The \emph{Newton polytope} $\go{N}$ is the region given by
\begin{equation}
	\go{N} = ConvexHull\left( \bigcup_{i=0}^N \{(p_i,r_i,q_i)+\R^3_+\}\right),
\end{equation}
that is, the unbounded convex polytope which is the convex hull of the union of the positive orthant $\R^3_+ = \{(p,r,q)\in \R^3|p,r,q\geq0\}$ translated to each point in the Newton diagram.
\end{defin}

The vertices $\go{V}$ of $\go{N}$ are a subset of the data; $\go{V}= \{(p_v,r_v,q_v)\}_{v=1}^L \subset \{(p_i,r_i,q_i)\}_{i=0}^N$. 

\smallskip

Fix $t$ with $0<|t|<1$. We examine the sum $|\sigma S|^2=\sum_{j=0}^N |t|^{2q_i}|x|^{2p_i}|y|^{2r_i}$ that appears in the denominator and as a factor of $\log|\sigma S|^2$ in the integrals under consideration. It is convenient to describe the regions in the $x,y$ variables where each particular term is \emph{dominant}; for fixed $x,y,t$, we say that the term $|t|^{2q_i}|x|^{2p_i}|y|^{2r_i}$ dominates the other terms in the sum if
\begin{equation}
	|t|^{2q_i}|x|^{2p_i}|y|^{2r_i} \geq |t|^{2q_j}|x|^{2p_j}|y|^{2r_j} \text{ for all $j\neq i$}.
\end{equation}
It is simple to describe the regions where the various terms dominate in terms of variables $\al,\beta$ where we set $|x| = |t|^\al$, $|y|=|t|^\beta$, where $\alpha,\beta$ are real and non-negative.  The term $|t|^{2q_i}|x|^{2p_i}|y|^{2r_i}$ dominates when
\begin{equation}
	|t|^{2(q_i + p_i\al+r_i\beta)} \geq |t|^{2(q_j + p_j\al+r_j\beta)}
\end{equation}
for all other points $(p_j,r_j,q_j)$ in the diagram, or
\begin{equation}\label{MvRegion}
	q_j +p_j\al+r_j\beta \geq q_i+p_i\al+r_i\beta, \qquad j \neq i.
\end{equation}
Together with the constraints $\al \geq 0$, $\beta \geq 0$, these inequalities describe the region $M_i$ where $|t|^{2q_i}|x|^{2p_i}|y|^{2r_i}$ dominates as the intersection of a set of half-planes in $(\al,\beta)$-space. $M_i$ describes the set of planes passing through the point lying below all the other points of the diagram, so it has empty interior unless the point $(p_i,r_i,q_i)$ is a vertex of the Newton polytope. To summarize, we have
\begin{lem} The positive quadrant $\al,\beta \geq 0$ is partitioned into polygonal regions $M_v$, $v=1,...,L$, on which the term $|t|^{2q_v}|x|^{2p_v}|y|^{2r_v}$ dominates. For each vertex with $(p_v,r_v) \neq (0,0)$, the region $M_v$ is the convex hull of the points $(m^x_{v,i},m^y_{v,i})$, where $i$ ranges over the number of faces of $\go{N}$ incident to the vertex, with normal vector $(m^x_{v,i},m^y_{v,i},1)$.
\end{lem} 

Note that the set of points $(m^x_{v,i},m^y_{v,i})$ forms a natural dual of the Newton polyhedron. Every Newton polytope for test configurations of this form also contains vertical faces $x=0$ and $y=0$, which do not play a role in the analysis.

\smallskip

In the interior of the region where $|t|^{2q_v}|x|^{2p_v}|y|^{2r_v}$ dominates, we may Taylor expand to obtain
\begin{align}\label{FirstOrderApprox}
	\log |\sigma S|^2 &= \log(|t|^{2q_v}|x|^{2p_v}|y|^{2r_v})+O(\frac{\sum_{j\neq v}
		|t|^{2q_j}|S_j|^2}{|t|^{2q_v}|x|^{2p_v}|y|^{2r_v}}),\\
	\frac{1}{|\sigma S|^2} &= \frac{1}{|t|^{2q_v}|x|^{2p_v}|y|^{2r_v}} +
	 		O(\frac{\sum_{j\neq v}|t|^{2q_j}|S_j|^2}{(|t|^{2q_v}|x|^{2p_v}|y|^{2r_v})^2}).\label{FirstOrderApprox2}
\end{align}
At the boundary of the region $M_v$, the term $|t|^{2q_v}|x|^{2p_v}|y|^{2r_v}$ shares the same order as other terms.

\subsection{Singular integral analysis}

We aim to calculate the divergent term in the Aubin-Yau functional as $|t| \rightarrow 0$, which is of the form $\log(1/|t|)$. Let $U$ be a small polydisk around the origin, which we may take to have radius $\epsilon$ in each variable. Subtracting off contributions of order $O(1)$, we calculate the contribution from the term $\int_U \phi \omega_\phi^2$ as follows:
\begin{align}
\int_U \phi \omega_\phi^2 &= \int_U \frac{\log|\sigma S|^2}{4\pi^2|\sigma S|^8} 
	\sum_{\{i,j,k,l\}} |t|^{2(q_i+q_j+q_k+q_l)}|x|^{2(p_i+p_j+p_k+p_l-1)}|y|^{2(r_i+r_j+r_k+r_l-1)}\nonumber\\
	&2D_4(i,j,k,l)\sqrt{-1}dx\wedge d\bar{x}\wedge \sqrt{-1}dy\wedge d\bar{y} + O(1)\\
	&= 8\int_0^\epsilon \int_0^\epsilon \sum_{\{i,j,k,l\}} D_4(i,j,k,l)\frac{|t|^{2(q_i+q_j+q_k+q_l)} x^{2(p_i+p_j+p_k+p_l)-1}y^{2(r_i+r_j+r_k+r_l)-1}}{|\sigma S|^8}\log |\sigma S|^2dxdy +O(1),
\end{align}
where we integrate out the angular variables. In the last line, at the risk of some confusion, we have allowed $x$ and $y$ to stand for the real absolute values of the complex numbers in the first line. We must be careful with all factors of 2 and $\pi$ in this calculation, since we must compare these localized integrals to the global contribution from the lowest weight. The factor of 4 comes from the change to polar coordinates in both $x$ and $y$ (as complex variables):
\begin{align}
\frac{\sqrt{-1}}{2\pi}dx\wedge d\bar{x} &= \frac{\sqrt{-1}}{2\pi}
	(du+\sqrt{-1}dv)\wedge(du-\sqrt{-1}dv)\\
	&= \frac{2}{2\pi}du\wedge dv\\
	&= \frac{2}{2\pi}rdr\wedge d\theta
\end{align}
Integrating each of the $\theta$ variables cancels a factor of $2\pi$ in denominator.

\smallskip

We set $$A(t) = 8\int_0^1 \int_0^1 \sum_{\{i,j,k,l\}} D_4(i,j,k,l)\frac{|t|^{2(q_i+q_j+q_k+q_l)} x^{2(p_i+p_j+p_k+p_l)-1}y^{2(r_i+r_j+r_k+r_l)-1}}{|\sigma S|^8}\log |\sigma S|^2dxdy$$ and set to determine the asymptotic behavior of $A(t)$. 

\smallskip

Our approach to dealing with the singular integrals depends on the following basic convergence result of Ermolaeva and Tsikh \cite{ET}. To state the theorem we need, let us quickly state two definitions. First, given a polynomial in $n$ variables $P(x_1,x_2,\dots,x_n)$, we define $\Delta(P)$ to be the convex hull in $\R^n$ of the points $\alpha_i$ in the non-negative integer lattice, where $P(x) = \sum_{\alpha \in \N^n} c_\alpha x^{\alpha}$, $c_\alpha\neq 0$. This convex region is also referred to as the \emph{Newton polygon} in the literature; it is different from what we have called the Newton polytope above, which treats the weights in $|t|$ differently, but hopefully the context and choice of symbol will make clear to what we are referring. 

\smallskip

Second, we say that a polynomial $Q(x)=\sum c_\alpha x^\alpha$ is \emph{quasi-elliptic} if each \emph{cut-off} $Q_a$ of $Q$, $$Q_a = \sum_{\alpha \in \Delta^a} c_\alpha x^\alpha,$$ where $$\Delta^a = \{k\in \Delta(Q): \langle a,k\rangle = \min_{l\in \Delta(Q)} \langle a,l\rangle\}$$ and $a$ is a covector in $\R^{n*}$, has the property that $Q_a$ does not vanish in $(\R \setminus\{0\})^n$. We will not be too concerned with this precise definition, since every polynomial that we consider here is a sum of terms with positive coefficients and even exponents in each variable, and so is certainly quasi-elliptic. 

\begin{theo}[\cite{ET}, Theorem 1]\label{RationalIntegralConvergence}
If $Q$ is a quasi-elliptic polynomial non-vanishing in $\R^n$, then the integral $$\go{I} = \int_{\R^n} \frac{P(x)}{Q(x)}dx$$ is absolutely convergent if and only if
\begin{equation}\label{convergenceCond}
I + \Delta(P) \subset \Delta^0(Q),
\end{equation}
that is, the translation of $\Delta(P)$ by $I = (1,\dots,1) \in \R^n$ lies in the interior $\Delta^0(Q)$ of $\Delta(Q)$.
\end{theo}

\begin{rem}In fact, the requirement that $Q$ be non-vanishing at the origin is unnecessary, since as long as $I + \Delta(P) \subset \Delta^0(Q)$, the integrand remains locally integrable in a neighborhood of the origin.
\end{rem}

\smallskip

Briefly let us recall the idea of the proof of Theorem \ref{RationalIntegralConvergence}. We embed $\R^n$ in a toroidal compactifaction $M_\Delta$ associated to the convex polyhedron $\Delta(Q)$. This amounts in essence to adjoining to $\R^n$ with local coordinates $(x_1,\dots,x_n)$ open neighborhoods $U_J$ with coordinates $(y_1,\dots,y_n)$ via the monomial transformation 
\begin{equation*}
x_i = y_1^{a^{j_1}_i}\dots y_n^{a^{j_n}_i} \qquad \text{for } i=1,\dots,n,
\end{equation*}
where $a^{j_i}\in \Z^n$, $J = (j_1,\dots,j_n)$ a subset of $\{1,2,\dots,N\}$, where $a^k$ is the smallest interior integer-coordinate normal vector to the $k$-th face of the $N$ faces of $\Delta(Q)$, and the $a^{j_i}$ form a basis of $\Z^n$. The condition (\ref{convergenceCond}) guarantees that in each coordinate neighborhood, under the given monomial change of variables, the integrand is bounded as the coordinates $y_i$ go to 0. It follows from the compactness of $M_\Delta$ that the integral is absolutely convergent.

\smallskip 

A technical point: it is possible that the Newton polygon of the polynomial $Q$ appearing in the denominator does not have the property that the normal vectors $a^{j_i}$ of faces meeting at a single vertex do not form a basis of $\Z^n$, for instance if fewer than $n$ faces meet at that vertex, or if the matrix whose columns are the $a^{j_i}$ is not unimodular. In this case the construction of a toric compactification of $\R^n$ is more complicated, but still possible, and Theorem \ref{RationalIntegralConvergence} still holds (see \cite{ET} and the references therein).

\smallskip

The integrals with which we are concerned also include a $\log$ factor in the integrand, but it also follows from the proof of \ref{RationalIntegralConvergence} and the finiteness of the improper integral $\int_0^1 |\log x| dx$ that we have

\begin{cor}\label{ConvergenceCor}
Let $Q$ be a quasi-elliptic polynomial, non-vanishing except possibly at 0, and $R$ a polynomial consisting of monomial terms with even exponents and positive coefficients. Then $$\go{J} = \int_{\R^n_+} \frac{P}{Q}\log R dx$$ is absolutely convergent if and only if (\ref{convergenceCond}) holds.
\end{cor}

Here is the crucial observation about the lowest order terms appearing in the energy functional that allows us to appeal to the convergence results above.

\begin{lem}\label{PointSelection}
Let $\{(p_i,r_i,q_i)\}_{i=0}^N$ be the data of a Newton diagram as above.
\begin{enumerate}
\item
Suppose $D_4(i,j,k,l)\neq 0$. Let $$P_t (x,y) = |t|^{2(q_i+q_j+q_k+q_l)}x^{2(p_i+p_j+p_k+p_l)-1}y^{2(r_i+r_j+r_k+r_l)-1}$$ and $Q_t(x,y) = \sum_{i=0}^N |t|^{2q_i}x^{2p_i}y^{2r_i}$. Then for each $t\neq 0$, as polynomials in $x,y$, $I + \Delta(P_t) \subset \Delta^0(Q_t^4)$. In particular, the integral over the first quadrant $\int_{\R^2_+}P_t/(Q_t)^4 < +\infty$ for each $t\neq 0$.
\item
Suppose $D_4(i,j,k,l) \neq 0$ and the points $\{(p_\beta,r_\beta,q_\beta)\}, \beta \in \{i,j,k,l\}$ all lie on the face $F$ of the Newton polyhedron. Then if $P = x^{2(p_i+p_j+p_k+p_l)-1}y^{2(r_i+r_j+r_k+r_l)-1}$ and $Q = \sum x^{2p_\alpha}y^{2r_\alpha}$, where the sum is over all indices $\alpha$ such that there exist points $\{(p_\alpha,r_\alpha,q_\alpha)\}$ in the Newton diagram lying on the face $F$, then $I + \Delta(P) \subset \Delta^0(Q^4)$.
\end{enumerate}
\end{lem}

\begin{proof}

The proof is the same in both cases: If $D_4(i,j,k,l)\neq 0$, then the 4 points $(p_i,r_i),(p_j,r_j),(p_k,r_k),(p_l,r_l)$ are not collinear, which ensures that these 4 points, chosen from among the lattice points in $\Delta(Q)$, which by definition sum to a point in $4\Delta(Q) = \Delta(Q^4)$, do not sum to a point on a boundary segment.
\end{proof}

\medskip

We are now ready to prove the main lemma in which we calculate the nontrivial contribution to the slope:
\begin{lem}\label{slopeCalc}
Let $\{i,j,k,l\}$ be a set of indices such that $D_4(i,j,k,l)\neq 0$. Suppose first that the points $\{(p_\beta,r_\beta,q_\beta)\}, \beta \in \{i,j,k,l\}$ all lie on the face $F$ of the Newton polyhedron. Fix $\ep>0$ small and set $$I(t)=\int_0^\ep \int_0^\ep \frac{|t|^{2(q_i+q_j+q_k+q_l)} x^{2(p_i+p_j+p_k+p_l)-1}y^{2(r_i+r_j+r_k+r_l)-1}}{|\sigma S|^8}\log |\sigma S|^2dxdy.$$ Let the equation of the affine plane of which $F$ is a subset be given by the equation $m_F^x x+m_F^y y + z = d_F$. Then	
	\begin{equation}
		I(t)= 2d_F \log|t|\int_0^\infty \int_0^\infty \frac{ x^{2(p_i+p_j+p_k+p_l)-1}y^{2(r_i+r_j+r_k+r_l)-1}}{(\sum_\al x^{2p_\al}y^{2r_\al})^4}dxdy +O(1)
	\end{equation}
where the sum in the denominator of the integrand is over all indices $\al$ where the points $(p_\al,r_\al,q_\al)$ lie on the face $F$ of $\go{N}$. If $D_4(i,j,k,l)\neq 0$ but the points $\{(p_\beta,r_\beta,q_\beta)\}, \beta \in \{i,j,k,l\}$ do not all lie on a single face of the Newton polyhedron, then $I(t) =  O(1)$.
\end{lem}

\begin{proof}
We compute the integral by rescaling. Let $x\rightarrow |t|^{m^x}x$, $y\rightarrow |t|^{m^y}y$, where $m^x,m^y\geq0$. The integral becomes
\begin{align}
	I(t)&=\int_0^\ep \int_0^\ep \frac{|t|^{2(q_i+q_j+q_k+q_l)} x^{2(p_i+p_j+p_k+p_l)-1}y^{2(r_i+r_j+r_k+r_l)-1}}{|\sigma S|^8}\log |\sigma S|^2dxdy \nonumber\\
	&=\int_0^{\ep |t|^{-m^x}} \int_0^{\ep |t|^{-m^y}}  \frac{|t|^{2((q_i+q_j+q_k+q_l)+m^x(p_i+p_j+p_k+p_l)+m^y(r_i+r_j+r_k+r_l))} x^{2(p_i+p_j+p_k+p_l)-1}y^{2(r_i+r_j+r_k+r_l)-1}}{(\sum_{a=0}^N |t|^{2(q_a+m^x p_a+m^y r_a)}x^{2p_a}y^{2r_a})^4}\nonumber\\
	&\log (\sum_{u=0}^N |t|^{2(q_u+m^x p_u+m^y r_u)}x^{2p_u}y^{2r_u})dxdy.
\end{align} 
Setting $\gamma = q_v+m^x p_v +m^y r_v = \min_u \{q_u+m^x p_u+m^y r_u\}$, we may factor $|t|^{2\gamma}$ out of each term in the denominator and the $\log$ factor. The integrand thus acquires an overall factor of $$|t|^Z =|t|^{2((q_i+q_j+q_k+q_l)+m^x(p_i+p_j+p_k+p_l)+m^y(r_i+r_j+r_k+r_l)-4\gamma)}.$$ Now the minimum $\gamma$ is realized at $q_v+m^x p_v +m^y r_v$ if $(p_v,r_v,q_v)$ is a vertex of $\go{N}$ and $(m^x,m^y)\in M_v$ by the convexity of $\go{N}$. We observe that the overall exponent $Z$ may be rewritten as
\begin{equation}
	Z = [(q_i-q_v)+m^x(p_i-p_v)+m^y(r_i-r_v)]+\cdots+[(q_l-q_v)+m^x(p_l-p_v)+m^y(r_l-r_v)] \geq 0.
\end{equation}
If this exponent is strictly greater than 0, then by Corollary \ref{ConvergenceCor} and Lemma \ref{PointSelection}, the integral $I(t)$ is O(1) as $t\rightarrow 0$. Since $q_v+m^x p_v +m^y r_v = \min_u \{q_u+m^x p_u+m^y r_u\}$, each of the four terms in $Z$ is non-negative, and equality is obtained only if each term in brackets is 0. In this case, the point $(m^x,m^y)$ must be contained in $M_i \cap M_j \cap M_k \cap M_l$, and since at least three of $i,j,k,l$ must be distinct, $(m^x,m^y)$ must lie at a corner of $M_v$, or in other words, each of the points $(p_i,r_i,q_i),...,(p_l,r_l,q_l)$ lies on a common face $F$ of $\go{N}$ with normal vector $(m^x,m^y,1)$, and $\gamma = d_F$. This shows that only sets of indices corresponding to four points lying on a single face of the Newton polytope contribute to the asymptotic slope, and moreover, that all higher-order terms in $\omega_0^i\wedge \omega_\phi^{n-i}$ do not contribute to the slope.

\smallskip

Therefore, supposing that $D_4(i,j,k,l)\neq0$ and the four indices corresponding to points on a face F of the Newton diagram, after rescaling as above, we have
\begin{align}
	I(t)= \int_0^{\ep |t|^{-m^x}} \int_0^{\ep |t|^{-m^y}} & \frac{x^{2(p_i+p_j+p_k+p_l)-1}y^{2(r_i+r_j+r_k+r_l)-1}}{(\sum_\al x^{2p_\al}y^{2r_\al} + |t|^\beta P(x,y,|t|))^4}\\
	&\big(\log |t|^{2\gamma}+\log(\sum_\al x^{2p_\al}y^{2r_\al} + |t|^\beta P(x,y,|t|))\big)dxdy.
\end{align}
where the sum is over all the indices $\al$ of points on the face $F$, $\beta>0$, and $P(x,y,|t|)$ is a polynomial. The effect of the rescaling is to pick out the terms that dominate in the denominator.  Now we take $2\gamma\log|t|$ outside the integral, and take the lowest order approximation of the integrand about $|t|=0$ as in (\ref{FirstOrderApprox}) and (\ref{FirstOrderApprox2}):
\begin{align}
I(t) = 2\gamma\log|t| & \int_0^{\ep |t|^{-m^x}} \int_0^{\ep |t|^{-m^y}}  \frac{x^{2(p_i+p_j+p_k+p_l)-1}y^{2(r_i+r_j+r_k+r_l)-1}}{(\sum_\al x^{2p_\al}y^{2r_\al})^4} dxdy\nonumber\\ 
 &+ \int_0^{\ep |t|^{-m^x}} \int_0^{\ep |t|^{-m^y}} O(\frac{|t|^\beta P(x,y)}{Q(x,y)})(1+\log R(x,y,|t|))dxdy
\end{align}
Letting $|t|$ go to 0, the limits of integration tend to $\infty$, and again by Corollary \ref{ConvergenceCor}, the second term goes to 0. 
\end{proof}

Now let us show there is no non-trivial contribution to the slope from the terms in the Aubin-Yau functional involving $\omega_0^2$ and $\omega_0\wedge\omega_\phi$; that is, their only contribution is from the lowest weight. For the integral $\int_X\log|\sigma S|^2\omega_0^2$ this is easy to see, since $\omega_0^2$ is bounded independent of $t$ and $\log|\sigma S|^2\leq\log|S_N|^2+c$ is integrable on $X$, therefore $\int_X\log|\sigma S|^2\omega_0^n\leq C=O(1)$. 

\smallskip

It remains to show $\int_X\log|\sigma S|^2\omega_0\wedge\omega_\phi$ is bounded as $|t|\rightarrow 0$. This can be seen by computing as before:
\begin{align}
	\int_U \log|\sigma S|^2 \omega\wedge\omega_\phi &= \int_U \frac{\log|\sigma S|^2}{4\pi^2|S|^4|\sigma S|^4} 
	\sum_{i,j,k,l} |t|^{2(q_i+q_j)}|x|^{2(p_i+p_j+p_k+p_l-1)}|y|^{2(r_i+r_j+r_k+r_l-1)}\nonumber\\
	&((ijkl)+(klij))\sqrt{-1}dx\wedge d\bar{x}\wedge \sqrt{-1}dy\wedge d\bar{y} +O(1)\\
	&=4\int_0^\infty \int_0^\infty \sum_{i,j,k,l}((ijkl)+(klij))\log |\sigma S|^2 \nonumber\\
	&\frac{|t|^{2(q_i+q_j)} x^{2(p_i+p_j+p_k+p_l)-1}y^{2(r_i+r_j+r_k+r_l)-1}}{|S|^4|\sigma S|^4}dxdy +O(1)
\end{align}

Under the scaling $x\rightarrow |t|^{m^x}x$, $y\rightarrow |t|^{m^y}y$, we may pull out a factor of $$|t|^{2((q_i+q_j)+m^x(p_i+p_j+p_k+p_l)+m^y(r_i+r_j+r_k+r_l) -2(q_\al+m^xp_\al+m^yr_\al))},$$
where $q_\al+m^xp_\al+m^yr_\al=\min_i\{q_i+m^xp_i+m^yr_i\}$. We find that the exponent
\begin{align}
	[(q_i-q_\al)+m^x(p_i-p_\al)&+m^y(r_i-r_\al)]+[(q_j-q_\al)+m^x(p_j-p_\al)+m^y(r_j-r_\al)]\nonumber\\
	+[m^xp_k+m^yr_k]&+[m^xp_l+m^yr_l]\geq0,
\end{align}
with equality only if each term in brackets is 0. But this can only happen when $(p_k,r_k)=(p_l,r_l)=(0,0)$, in which case $(ijkl)=(klij)=0$. It follows that the lowest order integral is a convergent integral multiplied by a positive power of $|t|$, which is $O(1)$ as $|t|\rightarrow 0$.

\smallskip

Thus we have proven Theorem \ref{lowestWeight}, and when combined with Lemma \ref{slopeCalc}, we obtain Theorem \ref{slope}.

\section{Examples}

The formula in (\ref{slope}) is most easily applicable in the case of toric surfaces. A polarized toric surface $(X,L)$ is associated to a polygon $P$ in the first quadrant of $\R^2$ with integral vertices including the point $(0,0)$. The lattice points $(p_i,r_i)$ of $\overline{P}$ are in 1-1 correspondence with a basis of sections $S_i = x^{p_i}y^{r_i}$, of $L$ in the coordinates of an open, dense subset of $X$. We must also be careful with the computation of our normalized volume. The volume of $X$, $V=\int_X \omega_0^2$ is related to the Euclidean volume of the polygon $P$ by
\begin{equation}
	Vol_{Euc}(P) = \int_X \frac{\omega_0^2}{2}=\frac{V}{2}.
\end{equation}

\subsection{Projective space $\PP^2$}

The bundle $\mathcal{O}(1)$ over $\PP^2$ is represented by a triangle with vertices $(0,0),(1,0),(0,1)$. The area of the triangle is 1/2, so $V=\int_X \omega_0^2=1$. We may specify a test-configuration or Bergman geodesic by assigning a non-negative weight over each point.  The lowest weight contribution is twice the average of the non-negative weights. The setup is symmetric with respect to the points $(1,0)$ and $(0,1)$, so there are not many essentially different configurations. Here are the possibilities:
\begin{enumerate}
\item The weight at $(0,0)$ is 0. In this case, the Newton polytope is trivial, consisting of the entire positive orthant, and there is no non-trivial contribution to the slope. The slope is positive and comes entirely from the lowest weight.

\item The weight at $(0,0)$ is greater than zero, which we may take to be 1 by the linear homogeneity of the slope in the weights. At least one of the remaining weights must be zero. We take our Newton diagram to be $\{(0,0,1),(0,1,q),(1,0,0)\}$. There are two possibilities for $q$:
\begin{enumerate}
\item $q>1$: In this case, the Newton polytope has only one non-trivial face, given by the equation $x+z=1$, and only the points $(0,0,1)$ and $(1,0,0)$ lie on it. There is no non-trivial contribution to the slope. The slope is equal to $\mu = 2(1+q)/3$.

\item $0\leq q \leq 1$: Again the Newton polytope consists of just the face $F:x+(1-q)y+z=1=d_F$. Now all three points lie on the face. We have $D_4(1,2,3,3)=1$. The slope is given by 
\begin{equation}
	\mu = \frac{2(1+q)}{3} - \frac{1}{3}16\cdot 1 \cdot 1 \cdot (I_1+I_2+I_3)
\end{equation}
where the three integrals come from the three choices of repeated index. We may compute
\begin{align}
	I_1 &= \int_0^\infty \int_0^\infty \frac{xy}{(1+x^2+y^2)^4}dxdy = \frac{1}{24},\\
	I_2 &= \int_0^\infty \int_0^\infty \frac{x^3y}{(1+x^2+y^2)^4}dxdy = \frac{1}{24},\\
	I_3 &= \int_0^\infty \int_0^\infty \frac{xy^3}{(1+x^2+y^2)^4}dxdy = \frac{1}{24},
\end{align}
so the slope comes to $\mu = 2(1+q)/3 -2/3=q/3$, which is 0 if $q=0$. In particular, $\{(0,0,1),(0,1,0),(1,0,0)\}$ is the configuration with the smallest slope, and it is non-negative.
\end{enumerate}
\end{enumerate}

\subsection{First Hirzebruch surface}

We may represent this toric surface as the convex polygon with vertices $\{(0,0),(2,0),(1,1),(0,1)\}$ with volume 3/2, and a Bergman geodesic is specified by a choice of weights for a Newton diagram $\{P_0:(0,0,q_{00}),P_1:(1,0,q_{10}),P_2:(2,0,q_{20}),P_3:(0,1,q_{01}),P_4:(1,1,q_{11})\}$. Now there are many more possibilities for the shape of the Newton polytope. Note that if, for example, $q_{00}=0$, $q_{10}=0$, or $q_{01}=0$, then the Newton polytope will be of the same types as for $\PP^2$.

\smallskip

Let us give two examples of test configurations. If we seek slopes that are as small as possible, we want the Newton diagrams to consist entirely of points on the boundary of the Newton polytope. In particular, if $q_{20}=0$, we must have $$q_{10} \leq \frac{1}{2} q_{00}.$$ The maximal number of faces each containing at least three points of the Newton polytope is three, and these may occur in two shapes: $F_1 =\{P_0,P_1,P_4\},F_2=\{P_0,P_3,P_4\},F_3=\{P_1,P_2,P_4\}$ or $\tilde{F}_1 =\{P_0,P_1,P_3\},\tilde{F}_2=\{P_1,P_3,P_4\},\tilde{F}_3=\{P_1,P_2,P_4\}$. The first case occurs if 
\begin{equation}\label{case1}
q_{10}+q_{01}-q_{00}-q_{11} >0,
\end{equation}
and the second case if the inequality is reversed.

\smallskip

Let us take as an example the first case. By solving the inequalities (\ref{MvRegion}), we obtain the equations of the faces of the polytope:
\begin{align}
	F_1 &=  \{(q_{00}-q_{10})p+(q_{10}-q_{11})r+q= q_{00}\}\\
	F_2 &= \{(q_{01}-q_{11})p+(q_{00}-q_{11})r+q = q_{00}\}\\
	F_3 &= \{q_{10}p+(q_{10}-q_{11})r+q = 2q_{10}\}
\end{align} 
The contribution from each face requires the evaluation of three integrals of the form $$I_{ijkl} = \int_0^\infty \int_0^\infty  \frac{x^{2(p_i+p_j+p_k+p_l)-1}y^{2(r_i+r_j+r_k+r_l)-1}}{(\sum_{\al\in F} x^{2p_\al}y^{2r_\al})^4}dxdy.$$ For example, on face $F_3$,
\begin{align}
	I_{1241} = \int_0^\infty \int_0^\infty  \frac{x^{9}y}{(x^2+x^2y^2+x^4)^4}dxdy = \frac{1}{24},
\end{align}
and in fact all of the integrals are the same as the integrals appearing in the $\PP^2$ calculation, and are all equal to $1/24$. Also, all the relevant factors $D_4(i,j,k,l)$ are equal to 1. For the total slope we have
\begin{align}
	\mu &= \frac{2(q_{00}+q_{10}+q_{01}+q_{11}+q_{20})}{5} - \frac{1}{3}\frac{1}{2\cdot 3/2}16\left( \frac{q_{00}}{8}+\frac{q_{00}}{8} + \frac{2q_{10}}{8}\right) \\
	&= \frac{-2(q_{00}+q_{10})+18(q_{01}+q_{11})}{45}
\end{align}
Combining the inequality (\ref{case1}) with $q_{10}<q_{00}/2$, we have $q_{01}>q_{00}/2+q_{11}$, so
\begin{align}
	\mu > \frac{-5q_{00}/2 + 18(q_{01}+q_{11})}{45}
	> \frac{13q_{00}/2 + 36q_{11}}{45}.
\end{align}

For the final example, suppose all of the points of the Newton diagram lie on a single face and $q_{20}=0$. Setting $q_{00}=1$, we have that the equation of the face must be $F: 1/2x + cy+z=1$, where $0\leq c\leq 1/2$. The sum $$m=\sum_{\{i,j,k,l\}^*}D_4(i,j,k,l)\int_0^\infty \int_0^\infty \frac{ x^{2(p_i+p_j+p_k+p_l)-1}y^{2(r_i+r_j+r_k+r_l)-1}}{(\sum_\al x^{2p_\al}y^{2r_\al})^4}dxdy$$ in (\ref{formula}) for the non-trivial part of the slope contains 32 terms: 5 choices for sets of 4 distinct indices, and 9 choices for sets of 3 indices with one repeated (the set $\{0,1,2\}$ excluded for being collinear), each of which gives three terms by the choice of the repeated index. The integrals in $m$ can be computed by Mathematica, for example: 
\begin{align}
	I_{0013} &= \int_0^\infty \int_0^\infty \frac{xydxdy}{(1+x^2+y^2+x^2y^2+x^4)^4} = \frac{7(-9+2\sqrt{3}\pi)}{648}\\
	I_{0014} &= \int_0^\infty \int_0^\infty \frac{x^3ydxdy}{(1+x^2+y^2+x^2y^2+x^4)^4} = \frac{6-\sqrt{3}\pi}{108}\\
	I_{0023} &= \int_0^\infty \int_0^\infty \frac{x^3ydxdy}{(1+x^2+y^2+x^2y^2+x^4)^4} = \frac{6-\sqrt{3}\pi}{108}\\
	I_{0024} &= \int_0^\infty \int_0^\infty \frac{x^5ydxdy}{(1+x^2+y^2+x^2y^2+x^4)^4} = \frac{-9+2\sqrt{3}\pi}{648}\\
	I_{0034} &= \int_0^\infty \int_0^\infty \frac{xy^3dxdy}{(1+x^2+y^2+x^2y^2+x^4)^4} = \frac{9-\sqrt{3}\pi}{324}\\
	I_{0113} &= \int_0^\infty \int_0^\infty \frac{x^3ydxdy}{(1+x^2+y^2+x^2y^2+x^4)^4} = \frac{6-\sqrt{3}\pi}{108}\\
	I_{0114} &= \int_0^\infty \int_0^\infty \frac{x^5ydxdy}{(1+x^2+y^2+x^2y^2+x^4)^4} = \frac{-9+2\sqrt{3}\pi}{648}\\
	I_{0123} &= \int_0^\infty \int_0^\infty \frac{x^5ydxdy}{(1+x^2+y^2+x^2y^2+x^4)^4} = \frac{-9+2\sqrt{3}\pi}{648}\\
	I_{0124} &= \int_0^\infty \int_0^\infty \frac{x^7ydxdy}{(1+x^2+y^2+x^2y^2+x^4)^4} = \frac{-9+2\sqrt{3}\pi}{648}\\
	I_{0133} &= \int_0^\infty \int_0^\infty \frac{xy^3dxdy}{(1+x^2+y^2+x^2y^2+x^4)^4} = \frac{9-\sqrt{3}\pi}{324}\\
	I_{0134} &= \int_0^\infty \int_0^\infty \frac{x^3y^3dxdy}{(1+x^2+y^2+x^2y^2+x^4)^4} = \frac{-36+7\sqrt{3}\pi}{648}\\
	I_{0144} &= \int_0^\infty \int_0^\infty \frac{x^5y^3dxdy}{(1+x^2+y^2+x^2y^2+x^4)^4} = \frac{45-8\sqrt{3}\pi}{648}\\
	I_{0223} &= \int_0^\infty \int_0^\infty \frac{x^7ydxdy}{(1+x^2+y^2+x^2y^2+x^4)^4} = \frac{-9+2\sqrt{3}\pi}{648}\\
	I_{0224} &= \int_0^\infty \int_0^\infty \frac{x^9ydxdy}{(1+x^2+y^2+x^2y^2+x^4)^4} = \frac{6-\sqrt{3}\pi}{108}\\
	I_{0233} &= \int_0^\infty \int_0^\infty \frac{x^3y^3dxdy}{(1+x^2+y^2+x^2y^2+x^4)^4} = \frac{-36+7\sqrt{3}\pi}{648}\\
	I_{0234} &= \int_0^\infty \int_0^\infty \frac{x^5y^3dxdy}{(1+x^2+y^2+x^2y^2+x^4)^4} = \frac{45-8\sqrt{3}\pi}{648}\\
	I_{0244} &= \int_0^\infty \int_0^\infty \frac{x^7y^3dxdy}{(1+x^2+y^2+x^2y^2+x^4)^4} = \frac{-36+7\sqrt{3}\pi}{648}\\
	I_{0334} &= \int_0^\infty \int_0^\infty \frac{xy^5dxdy}{(1+x^2+y^2+x^2y^2+x^4)^4} = \frac{27-4\sqrt{3}\pi}{216}\\
	I_{0344} &= \int_0^\infty \int_0^\infty \frac{x^3y^5dxdy}{(1+x^2+y^2+x^2y^2+x^4)^4} = \frac{-9+2\sqrt{3}\pi}{216}\\
	I_{1123} &= \int_0^\infty \int_0^\infty \frac{x^7ydxdy}{(1+x^2+y^2+x^2y^2+x^4)^4} = \frac{-9+2\sqrt{3}\pi}{648}\\
	I_{1124} &= \int_0^\infty \int_0^\infty \frac{x^9ydxdy}{(1+x^2+y^2+x^2y^2+x^4)^4} = \frac{6-\sqrt{3}\pi}{108}\\
	I_{1134} &= \int_0^\infty \int_0^\infty \frac{x^5y^3dxdy}{(1+x^2+y^2+x^2y^2+x^4)^4} = \frac{45-8\sqrt{3}\pi}{648}\\
	I_{1223} &= \int_0^\infty \int_0^\infty \frac{x^9ydxdy}{(1+x^2+y^2+x^2y^2+x^4)^4} = \frac{6-\sqrt{3}\pi}{108}\\
	I_{1224} &= \int_0^\infty \int_0^\infty \frac{x^{11}ydxdy}{(1+x^2+y^2+x^2y^2+x^4)^4} = \frac{7(-9+2\sqrt{3}\pi)}{648}\\
	I_{1233} &= \int_0^\infty \int_0^\infty \frac{x^5y^3dxdy}{(1+x^2+y^2+x^2y^2+x^4)^4} = \frac{45-8\sqrt{3}\pi}{648}\\
	I_{1234} &= \int_0^\infty \int_0^\infty \frac{x^7y^3dxdy}{(1+x^2+y^2+x^2y^2+x^4)^4} = \frac{-36+7\sqrt{3}\pi}{648}\\
	I_{1244} &= \int_0^\infty \int_0^\infty \frac{x^9y^3dxdy}{(1+x^2+y^2+x^2y^2+x^4)^4} = \frac{9-\sqrt{3}\pi}{324}\\
	I_{1334} &= \int_0^\infty \int_0^\infty \frac{x^5y^3dxdy}{(1+x^2+y^2+x^2y^2+x^4)^4} = \frac{45-8\sqrt{3}\pi}{648}\\
	I_{1344} &= \int_0^\infty \int_0^\infty \frac{x^5y^5dxdy}{(1+x^2+y^2+x^2y^2+x^4)^4} = \frac{-9+2\sqrt{3}\pi}{216}\\
	I_{2234} &= \int_0^\infty \int_0^\infty \frac{x^9y^3dxdy}{(1+x^2+y^2+x^2y^2+x^4)^4} = \frac{9-\sqrt{3}\pi}{324}\\
	I_{2334} &= \int_0^\infty \int_0^\infty \frac{x^5y^5dxdy}{(1+x^2+y^2+x^2y^2+x^4)^4} = \frac{-9+2\sqrt{3}\pi}{216}\\
	I_{2344} &= \int_0^\infty \int_0^\infty \frac{x^7y^5dxdy}{(1+x^2+y^2+x^2y^2+x^4)^4} = \frac{27-4\sqrt{3}\pi}{216}.
\end{align}
More detail on the evaluation of such period integrals will appear in forthcoming work. (Note, for example, the symmetry in the above values where if $$I(P,R) = \int_0^\infty \int_0^\infty \frac{x^Py^Rdxdy}{(1+x^2+y^2+x^2y^2+x^4)^4},$$ then $I(P,R) = I(13/2-R/2-P,R)$.) Incredibly, the overall sum $m$ of these integrals weighted by the numbers $D_4(i,j,k,l)$ is rational: $m=3/8$. The total asymptotic slope is 
\begin{align}
	\mu &= \frac{2(q_{00}+q_{10}+q_{01}+q_{11}+q_{20})}{5}-\frac{1}{3}\frac{1}{2\cdot3/2}16q_{00}\frac{3}{8}\\
		&= \frac{2(q_{10}+q_{01}+q_{11}+q_{20})}{5}-\frac{2q_{00}}{3}.
\end{align} 
Setting $q_{00}=1$, $\mu$ attains its smallest value when $q_{10}=q_{01}=1/2$, $q_{20}=q_{11}=0$, in which case $\mu=-4/15$, and the configuration is unstable. 

\section{Directions for further work}

We hope to be able to extend this analysis both to higher dimension and to nontrivial non-toric examples.  With regard to the questions posed in the introduction, we might proceed by analyzing the asymptotic slope formula to identify the configuration of weights that gives the minimum slope, and perhaps characterize the data that admits a negative minimum. Given the examples above and the inclusion of many more terms to the nontrivial part of the slope, it is tempting to conjecture that the minimum slope is achieved by weight configurations that are \emph{balanced} in the sense of having all points lying on a single face of the Newton polytope, and minimal among such configurations. Of course, in dimensions greater than one, determining the minimum slope is made more complicated by the appearance of the period integrals in the formula. If we can identify the directions of negative asymptotic slope, it will hopefully shed light on the singularities of solutions of the Fano KE equation.


\begin{thebibliography}{HHHH}

\bibitem{Berm} R. Berman, $K$-Polystability of Q-Fano varieties admitting K\"{a}hler-Einstein metrics, arXiv:1205.6214v2.


\bibitem{Bern} B. Berndtsson, Convexity on the space of \kahler metrics.  Ann. Fac. Sci. Toulouse Math. (6) 22 (2013), no. 4, 713–746.

\bibitem{CDS1} X. Chen, S.K. Donaldson, S. Sun, K\"{a}hler-Einstein metrics on Fano manifolds. I: Approximation of metrics with cone singularities, J. Amer. Math. Soc. {\bf 28} (2015), no. 1, 183-197.

\bibitem{CDS2} X. Chen, S.K. Donaldson, S. Sun, K\"{a}hler-Einstein metrics on Fano manifolds. II: Limits with cone angle less than $2\pi$. J. Amer. Math. Soc. {\bf 28} (2015), no. 1, 199-234.

\bibitem{CDS3} X. Chen, S.K. Donaldson, S. Sun, K\"{a}hler-Einstein metrics on Fano manifolds. III: Limits as cone angle approaches $2\pi$ and completion of the main proof. J. Amer. Math. Soc. {\bf 28} (2015), no. 1, 235-278. 


\bibitem{D1} S.K. Donaldson, Anti self-dual Yang-Mills connections over complex algebraic surfaces and stable vector bundles, Proc. London Math. Soc. \textbf{50} (1985), 1-26.

\bibitem{D} S.K. Donaldson, Scalar curvature and projective embeddings I, J. Differential Geom. {\bf 59} (2001) 479-522.

\bibitem{D2} S.K. Donaldson, Scalar curvature and stability of toric varieties, J. Differential Geom. {\bf 59} (2002) 289-349.

\bibitem{DS} S.K. Donaldson and S. Sun, Gromov-Hausdorff limits of \kahler manifolds and algebraic geometry, arXiv:1206.2609

\bibitem{ET} T.O. Ermolaeva and A.K. Tsikh, Integration of rational functions over $\R^n$ by means of toric compactifications and multidimensional residues, Sbornik: Mathematics. {\bf 187}:9 (1996) 1301-1318.

\bibitem{KZ} S. Klevtsov and S. Zelditch, Stability and integration over Bergman metrics, J. High Energy Phys. no. 7 (2014) 100

\bibitem{P} S.T. Paul, Geometric analysis of Chow-Mumford stability, Adv. Math. 182 (2004), no. 2, 333-356.

\bibitem{PhStein} D.H. Phong and E.M. Stein, The Newton polyhedron and oscillatory integral operators, Acta Math. {\bf 179} (1997) 105-152.

\bibitem{PSAlg} D.H. Phong and J. Sturm, Algebraic estimates, stability of local zeta functions, and uniform estimates for distribution functions, Ann. of Math. {\bf 152} (2000) 277-329.

\bibitem{PSStab} D.H. Phong and J. Sturm, Stability, energy functionals, and K\"ahler-Einstein metrics, Commun. Analysis and Geometry, {\bf 11} (2003) 565-597.

\bibitem{PSMab} D.H. Phong and J. Sturm, On asymptotics for the Mabuchi energy functional, Abstract and applied analysis, 271-286, World Sci. Publ., River Edge, NJ, 2004.

\bibitem{PSLSC} D.H. Phong and J. Sturm, Lectures on stability and constant scalar curvature, Handbook of Geometric Analysis, No. 3, 357-436, Adv. Lect. Math. (ALM), 14, Int. Press, Somerville, MA, 2010.




\bibitem{T2} G. Tian, $K$-stability and K\"{a}hler-Einstein metrics, arXiv:1211.4669

\bibitem{UY} K. Uhlenbeck, S.-T. Yau, On the existence of Hermitian-Yang-Mills connections in stable vector bundles, Comm. Pure Appl. Math. \textbf{39-S} (1986), 257-293.

\bibitem{Y1} S.-T. Yau, On the Ricci curvature of a compact \kahler manifold and the complex Monge-Amp\`{e}re equation, I, Comm. Pure Appl. Math. \textbf{31} (1978), no. 3, 339-411.

\bibitem{Y} S.-T. Yau, Open problems in geometry, Proc. Symposia Pure Math. {\bf 54} (1993) 1-28.

\bibitem{Z} S.W. Zhang, Heights and reductions of semistable varieties, Compositio Math. {\bf 104} (1996) 77-105.

\end{thebibliography}
\end{document}